\documentclass[11pt]{amsart}
\usepackage{amsaddr}
\newtheorem{conj}{Conjecture}[section]
\newtheorem{prop}[conj]{Proposition}
\newtheorem{thm}[conj]{Theorem}
\newtheorem{lem}[conj]{Lemma}
\newtheorem{cor}[conj]{Corollary}
\theoremstyle{definition}
\newtheorem{remark}[conj]{Remark}
\newtheorem{example}[conj]{Example}

\usepackage[utf8]{inputenc}
\usepackage{enumerate}
\usepackage[hidelinks]{hyperref}
\usepackage[all]{xy}

\DeclareFontFamily{OT1}{pzc}{}
\DeclareFontShape{OT1}{pzc}{m}{it}{<-> s * [1.10] pzcmi7t}{}
\DeclareMathAlphabet{\mathpzc}{OT1}{pzc}{m}{it}

\renewcommand{\vee}{*}

\newcommand{\q}{\mathfrak q}
\newcommand{\rr}{\mathfrak r}
\newcommand{\one}{\mathbf{1}}
\newcommand{\OO}{\mathcal O}
\newcommand{\PP}{\mathbb P}
\newcommand{\HH}{\mathbb H}
\newcommand{\QQ}{\mathbb Q}
\newcommand{\CCC}{\bar{\mathfrak C}}
\newcommand{\vi}{\mathbf i}

\DeclareMathOperator{\Aut}{Aut}
\DeclareMathOperator{\todd}{todd}
\DeclareMathOperator{\ext}{ext}
\DeclareMathOperator{\ch}{ch}

\DeclareMathOperator{\Quot}{Quot}

\DeclareMathOperator{\wt}{wt}
\begin{document}

\title{Universal Series for Hilbert Schemes and Strange Duality}
\author{Drew Johnson}
\address{Department of Mathematics, University of Oregon, Eugene}
\email{drewj@uoregon.edu}

\begin{abstract}
	We show how the ``finite Quot scheme method" applied to Le Potier's strange duality on del Pezzo surfaces leads to conjectures (valid for all smooth complex projective surfaces) relating two sets of universal power series on Hilbert schemes of points on surfaces: those for top Chern classes of tautological sheaves, and those for Euler characteristics of line bundles. We have verified these predictions computationally for low order. We then give an analysis of these conjectures in small ranks. We also give a combinatorial proof of a formula predicted by our conjectures: the top Chern class of the tautological sheaf on $S^{[n]}$ associated to the structure sheaf of a point is equal to $(-1)^n$ times the $n$th Catalan number.
\end{abstract}
\maketitle
\section{Introduction}
\subsection{Universal Series for Tautological Bundles on Hilbert Schemes}
Let $S$ be a smooth, projective surface over $\mathbb C$. The Hilbert scheme $S^{[n]}$ parametrizes unordered sets of $n$ points on $S$, as well as the possible non-reduced scheme structures that can occur when points collide. Given a vector bundle $F$ (or more generally, a class in the Grothendiek group $K_0$) on $S$, one can form the so-called \emph{tautological sheaf} $F^{[n]}$ on $S^{[n]}$ (see Section \ref{sec:taut-sheaves}). An important result of Ellingsrud, G\"{o}ttsche, and Lehn in \cite{ellingsrud_cobordism_2001} is that any Chern number formed from $F^{[n]}$ and the tangent bundle of $S^{[n]}$ is in fact a polynomial, independent of $S$, of the Chern numbers of $F$ and $S$. One can say more about the structure of these polynomials. We recall two important special cases here. 


\begin{thm}[\cite{ellingsrud_cobordism_2001}, Theorem 5.3] \label{thm:chi-series}
	For any surface $S$ and class $F$ in $K_0(S)$ of rank $r$, we have
	\[
	\sum_{n=1}^\infty \chi\left(\det F^{[n]}\right)z^n = g_r(z)^{\chi(\det F)} \cdot f_r(z)^{\frac{1}{2}\chi(\OO_S)} \cdot A_r(z)^{c_1(F).K_S-\frac{1}{2}K_S^2}\cdot B_r(z)^{K_S^2},
	\]
	where $A_r(z)$, $B_r(z)$, $f_r(z)$, $g_r(z)$ are power series in $z$ depending only on $r$, and
	\begin{align*}
	f_r(z) &= \sum_{k\ge0} \binom{(1-r^2)(k-1)}{ k} z^k \\
	g_r(z) &= \sum_{k\ge0} \frac{1}{1-(r^2-1)k}\binom{1-(r^2-1)k}{k} z^k.
	\end{align*}
	Furthermore $A_r(z)=B_r(z)=1$ for $r=-1,0,1$.
\end{thm}
Explicit formulas for $A_r(z)$ and $B_r(z)$ are not know, even conjecturally, although they can be computed for low order. We remark that any line bundle on $S^{[n]}$ can be written as $\det F^{[n]}$, so in fact this theorem contains the Euler characteristics of all line bundles on Hilbert schemes.

Next, the following can be deduced from Theorem 4.2 of \cite{ellingsrud_cobordism_2001}:
\begin{thm} \label{thm:c2n-series} For any surface $S$ and any class $F \in K_0(S)$ of rank $s$, we have
\[
\sum_{n=0}^{\infty}  c_{2n}(F^{[n]}) w^n= V_s(w)^{c_2(F)} \cdot W_s(w)^{\chi(\det F)}\cdot X_s(w)^{\frac12\chi(\mathcal O_S)} \cdot Y_s(w)^{c_1(F).K_S - \frac12{K^2}} \cdot Z_s(w)^{K_S^2} 
\]
where $V_s(w),W_s(w),X_s(w),Y_s(w),Z_s(w)$ are power series in $w$ depending only on $s$.
\end{thm}
In Theorem \ref{thm:c2n-series}, none of the series are known in general.

One might ask whether there is any relationship between the series in Theorems \ref{thm:chi-series} and \ref{thm:c2n-series}. This paper will explain how such a relationship is suggested by Le Potier's strange duality for surfaces.

In Section \ref{sec:sd}, we give the background for strange duality and the ``finite Quot scheme method" which was first used by Marian and Oprea for curves in \cite{marian_level-rank_2007} and later used for surfaces by Bertram, Goller, and the author in \cite{bertram_potiers_2016}. 
In Section \ref{sec:strange-hilbert} we specialize to the case where one of the moduil spaces is $S^{[n]}$. Here the finite Quot scheme method compares the length of a certain finite Quot scheme on $S$ to the dimension of the space of sections of a line bundle on $S^{[n]}$. We explain how the length of the finite Quot scheme can be viewed as a $2n$-codimensional Chern class of a certain tautological bundle. From this, we derive in Section \ref{sec:the-conj} the following:

\begin{conj} \label{conj:series}
Let $\phi_s(w) = V_s(w)^{2-s}$. Define a change of variables between $z$ and $w$ via $w=z\phi_s(w)$, and also let $s = r+1$. Then we have the following: 

\begin{align}
g_r(z) &= V_s(w)\cdot W_s(w) \tag{C1}\label{C1}\\
f_r(z) &= \frac{X_s(w)}{\phi_s(w)^4\left(\frac{dz}{dw}\right)^2} \tag{C2}\label{C2}\\ 
A_r(z) &= Y_s(w)\tag{C3} \label{C3} \\
B_r(z) &= Z_s(w) \tag{C4}\label{C4}
\end{align}
\end{conj}

In Section \ref{sec:small-rank}, we look at some special cases for small values of $s$ and $r$ and show how they support our conjecture. Most notably, we show how our conjecture is consistent with the computation of the top Segre class of tautological sheaves associated to line bundles by Marian, Oprea, and Pandharipande in  \cite{marian_segre_2015}.

Finally, we  use the combinatorial description by Lehn \cite{lehn_chern_1999} of the cohomology of $S^{[n]}$ in terms of $\q$-operators to prove  a novel formula predicted by our conjectures. We obtain in Theorem 6.1 that for any smooth projective surface $S$
\[
c_{2n}(\mathcal O_p^{[n]}) = (-1)^n C_n
\]
where $\OO_p^{[n]}$ is the tautological sheaf  on $S^{[n]}$ associated to $\OO_p$, the structure sheaf of a point, and $C_n$ is the $n$th Catalan number.

As additional support for our conjecture, we have used Sage \cite{sage} to  implement an algorithm based on  \cite{lehn_chern_1999} and computed the series of Theorem \ref{thm:c2n-series} up to order six. The series of Theorem \ref{thm:chi-series} were first computed in \cite{ellingsrud_cobordism_2001}, and to higher order in \cite{bertram_potiers_2016} and \cite{johnson_two_2016}. The computations agree with Conjecture \ref{conj:series}. 

\subsection{Further Progress} Using a preprint of this work and the accompanying numerical data, Marian, Oprea, and Pandharipande have obtained predictions for $V_s(w)$, $W_s(w)$, and $X_s(w)$ that are consistent with Conjecture \ref{conj:series}. See \cite{marian_combinatorics_2017}.

\section{Strange Duality} \label{sec:sd}
\subsection{Set up} \label{sec:set-up}

Let $S$ be a smooth, projective del Pezzo surface over $\mathbb C$ with canonical 
class $K_S$. Let
$$ e,f \in H^*(S,{\mathbb Q}) =  H^0(S,{\mathbb Q}) \oplus H^2(S,{\mathbb Q}) \oplus H^4(S,{\mathbb Q}) $$ 
be cohomology classes that are orthogonal with respect to the {\it Mukai pairing}
\[
\chi(e,f) = \int_S e^\vee \cup f \cup \todd(S),
\]
where we write $e = (e_0,e_1,e_2)$, $e^\vee = (e_0,-e_1,e_2)$, and 
$\todd(S) = (1,-{K_S}/2,1)$ is the Todd class of  $S$. Using this pairing, the Hirzebruch-Riemann-Roch Theorem implies that for $E$ and $F$ coherent sheaves on $S$, the Chern characters pair as
\[
\chi\left(\ch(E),\ch(F)\right) = \chi(E,F) = \sum (-1)^i\ext^i(E,F).
\]

Assuming that the moduli spaces $M_S(e^{\vee})$ and $M_S(f)$ of Gieseker-semistable coherent sheaves are non-empty, the orthogonality of $e$ and $f$, together with some other mild conditions, implies that the ``jumping locus''
\[
\Theta = \{\, (\hat{E},F) \mid h^0(\hat{E} \otimes F) > 0 \,\} \subset M_S(e^{\vee}) \times M_S(f)
\]
has the structure of a Cartier divisor (see \cite{le_potier_dualite_2005}, \cite{scala_dualite_2007}, \cite{danila_resultats_2002}). One also has divisors 
\[
\Theta_F = \{[\hat E]: h^0(\hat E \otimes F) >0\} \subset M_S(e^\vee), \;\;\;\; \Theta_{\hat E} = \{[F]:h^0(\hat E \otimes F) >0\} \subset M_S(f). 
\] 
As the Picard group of $S$ is discrete, the $\Theta_F$ are all linearly equivalent as $F$ varies in $M_S(f)$; we refer to the associated line bundle as $\mathcal O(\Theta_f)$ and similarly for $\mathcal O(\Theta_{e^\vee})$.
One can then check that the  line bundle associated to $\Theta$ satisfies
\[
 {\mathcal O}_{M_S(e^{\vee})\times M_S(f)}(\Theta) = \pi_1^*{\mathcal O}_{M_S(e^{\vee})}(\Theta_f)\otimes \pi_2^*{\mathcal O}_{M_S(f)}(\Theta_{e^{\vee}})
\]
Thus
$$H^0\big(M_S(e^{\vee}) \times M_S(f), {\mathcal O}(\Theta)\big) = H^0\big(M_S(e^{\vee}),{\mathcal O}(\Theta_f)\big) \otimes 
H^0\big(M_S(f),{\mathcal O}(\Theta_{e^{\vee}})\big),$$
so a section defining $\Theta$ determines a map (well-defined up to a choice of a (non-zero) scalar):
\begin{equation}
\operatorname{SD}_{e,f}: H^0\big(M_S(f),{\mathcal O}(\Theta_{e^{\vee}})\big)^* \rightarrow H^0\big(M_S(e^{\vee}),{\mathcal O}(\Theta_f)\big) \label{eq:sd}
\end{equation}
One could interpret $\operatorname{SD}_{e,f}$ as taking the hyperplane in $H^0\big(M_S(f),{\mathcal O}(\Theta_{e^{\vee}})\big)$ determined by a point $[F] \in M_S(f)$ to the section (up to scaling) $\Theta_F \in H^0\big(M_S(e^{\vee}),{\mathcal O}(\Theta_f)\big)$. 

\begin{conj}[Le Potier's Strange Duality] $\operatorname{SD}_{e,f}$ is an isomorphism.
\end{conj}

\subsection{Finite Quot Schemes} \label{sec:finite-quot}

The finite Quot scheme idea was originally in \cite{marian_level-rank_2007}. The argument is as follows.   

Let $v = e + f$ be the Chern character of a direct sum $E \oplus F$ of coherent sheaves of orthogonal Chern characters $e$ and $f$, and suppose 
$V$ is a coherent sheaf with $\ch(V) = v$. 
Then each element of the Grothendieck quot scheme $\Quot(V,f)$ of coherent sheaf quotients of $V$ of class $f$ gives rise to an exact sequence
\begin{equation}
0 \rightarrow E \rightarrow V \rightarrow F \rightarrow 0 \label{eq:evf}
\end{equation}
and $\chi(E,F) = 0$ is the expected dimension of $\Quot(V,f)$. Moreover, if
$$\hom(E,F) = \ext^1(E,F) = \ext^2(E,F) = 0$$
then the point of $\Quot(V,f)$ corresponding to \eqref{eq:evf} is isolated and reduced. 

Now suppose a sufficiently general $V$ may be chosen so that $\Quot(V,f)$ is finite and that each quotient
\begin{equation}
0 \rightarrow E_i \rightarrow V \rightarrow F_i \rightarrow 0 \label{eq:evfi}
\end{equation}
has the property that $E_i$ and $F_i$ are Gieseker-semistable, that $\ext^j(E_i,F_i) = 0$ for all $i,j$, and that $E_i$ is locally free.\footnote{Local freeness will be automatic since $V$ will be locally free and the $F_i$ will be torsion free.}  The condition $\operatorname{Hom}(E_i, F_i) =0$ says that $\Theta_{F_i}([E_i^\vee])$ does not vanish. Furthermore, by composing maps from \eqref{eq:evfi} (use stability) we see that $\hom(E_i, F_j) \neq 0$ for $i \neq j$.  This shows that $\Theta_{F_i}([E_j^\vee])$ does vanish. Hence the ``matrix'' $\Theta_{F_i}([E_j^\vee])$ is diagonal with non-zero entries on the diagonal, and we conclude that the sections $\Theta_{F_i}$ are linearly independent. As we observed above, sections $\Theta_{F}$ are in the image of $\operatorname{SD}_{e,f}$ for any $F$. Hence, if the length of the Quot scheme is equal to the dimension of $H^0\big(M_S(e^{\vee}),{\mathcal O}(\Theta_f)\big)$, we can conclude that $\operatorname{SD}_{e,f}$ is surjective, and if the length of the Quot scheme is equal to the dimension of $H^0\big(M_S(f),{\mathcal O}(\Theta_{e^{\vee}})\big)^*$, we can conclude that $\operatorname{SD}_{e,f}$ is injective.

\section{Hilbert Schemes and Strange Duality} \label{sec:strange-hilbert}
Unfortunately, the dimensions of both sides of \eqref{eq:sd} are unknown in general. One special case where we can compute a dimension is when $f = (1,0,-n)$, the Chern character of the ideal sheaf ${\mathcal I}_Z \subset {\mathcal O}_S$
of a zero-dimensional length $n$ subscheme $Z \subset S$. In this case, we may identify
$$M_S(f) = S^{[n]}$$
with the Hilbert scheme. One can use the algorithm described in \cite{ellingsrud_cobordism_2001} to compute a finite number of terms of the series of Theorem \ref{thm:chi-series} and obtain the Euler characteristics.  Of course, to do this we need to be able to write $\Theta_{e^\vee}$ in the form $\det F^{[n]}$ for some $F$, which we do in Section \ref{sec:theta-hilbert}.

Using the finite Quot scheme method in this context means that we are counting sequences \eqref{eq:evf} where the quotient is an ideal sheaf of $n$ points.  Such quotients may be viewed as sections of the dual bundle $V^*$ that vanish at $n$ points. In \cite{bertram_potiers_2016}, we found an ``expected" length of $\Quot\big(V,(1,0,-n)\big)$ 
by interpreting such sections as multiple points of a map from an auxiliary variety (obtained from $V^*$) to a projective space.  The number of multiple points can be computed for $n \le 7$ using methods of Rim\`{a}nyi and Marangell \cite{marangell_general_2010} and matches the Euler characteristic of $\Theta_{e^\vee}$.

In Section \ref{sec:length-taut}, we explain a different way to view the length of the finite Quot scheme which is a key idea of this paper.

\subsection{Tautological Sheaves} \label{sec:taut-sheaves}
To a vector bundle $F$ on $S$, one can associate a series of ``tautological" vector bundles on the Hilbert schemes $S^{[n]}$, as follows. Consider the standard diagram
\[
\xymatrix{
	\mathcal Z_n \ar@{^(->}[r]& S^{[n]} \ar[r]^(.6){q}  \ar[d]^p \times S & S \\
	    & S^{[n]}
	}
\]
where $\mathcal Z_n$ is the universal subscheme. Then
\begin{equation}
F^{[n]} := p_*(\OO_{\mathcal Z_n} \otimes q^* F). \label{eq:tautsheaf}
\end{equation}
If $Z$ is a subscheme of $n$ isolated points on $S$, then one can identify the fiber of $F^{[n]}$ at the point of $S^{[n]}$ corresponding to $Z$ with the sum of the fibers of $F$ at the points of $Z$.
If $F$ has rank $s$, then $F^{[n]}$ has rank $ns$. This construction can be extended to $K_0(S)$ by applying it to locally free resolutions.

\subsection{Theta Bundles on Hilbert Schemes} \label{sec:theta-hilbert}
In order to make use of Theorem \ref{thm:chi-series}, we need to write $\Theta_{e^\vee}$ in the form $\det F^{[n]}$ for some $F$.

The class of the determinant line bundle induced by a sheaf $\hat E$ of class $e^\vee$ is:  
\[
\Theta_{\hat E}=-c_1(Rp_*(q^*{\hat E} \otimes I_{\mathcal Z_n})).
\]
Here $I_{\mathcal Z_n}$ is the universal ideal sheaf on $S^{[n]} \times S$. 

Starting with the exact sequence 
\[
0 \rightarrow I_{\mathcal Z_n} \rightarrow \OO_{S \times S^{[n]}} \rightarrow \OO_{\mathcal Z_n} \rightarrow 0
\]
on $S \times S^{[n]}$, one then tensors by $q^*\hat E$, and then applies $Rp_*$, obtaining an exact triangle
\[
	Rp_*( q^*\hat E \otimes I_{\mathcal Z_n}  ) \rightarrow H^0(S, \hat E) \otimes \OO_{S^{[n]}} \rightarrow \hat E^{[n]} \rightarrow \cdots
\]
The line bundle $\OO(\Theta_{e^\vee})$ is the dual of the determinant of the first term in this triangle. Since the middle term is a trivial bundle, we obtain: 
\[
\OO(\Theta_{e^\vee}) = \det \hat E^{[n]}. 
\]

\subsection{Lengths of Quot schemes as Chern classes of tautological bundles} \label{sec:length-taut}

Fix $n\ge 1$ and let $e_n = (r, -L, \ch_2(e_n))$. The second Chern charater $\ch_2(e_n)$ is determined by the condition $\chi(e_n, (1,0,-n)) = 0$. Let $v_n= e_n + (1,0-n)$. We have $\det v^\vee_n = \det e^\vee_n = L$ (by a slight abuse we use $L$ for both the first chern class and the associated line bundle), and calculations with the Hirzebruch-Riemann-Roch Theorem yield:
\begin{align}
\chi(v^\vee_n) =n(r+1)-2n+1 \label{eq:chiV*}\\
c_2(v^\vee_n) = \chi(L)-(n-1)(r-1) \label{eq:c2V*}
\end{align}

Now, letting $V_n$ be a vector bundle with chern character $v_n$, we start with the map on $S^{[n]} \times S$
\[
 \OO_{S^{[n]} \times S} \rightarrow \OO_{Z_n},
\]
tensor by $q^*V^*_n$, and push forward by $p$ to obtain a map
\begin{equation}
H^0(S,V^*_n)\otimes \OO_{S^{[n]}} \rightarrow (V^*_n)^{[n]}. \label{eq:V*sec}
\end{equation}
Consider a section $\sigma$ of $V^*_n$, which corresponds via \eqref{eq:V*sec} to a section $\sigma^{[n]}$ of $(V^*_n)^{[n]}$.
The vanishing of $\sigma^{[n]}$ at  a point $[Z] \in S^{[n]}$ corresponds to the vanishing of $\sigma$ along $Z \subset S$. This in turn corresponds to a map $\sigma^\vee:V_n \rightarrow I_Z$, which we are trying to count. Hence we obtain an expected  length of  $\Quot(V_n,(1,0,-n))$ as the number of sections $\sigma^{[n]}$ vanishing at a point, or, equivalently, the number of points where \eqref{eq:V*sec} drops rank.

Since $V^*_n$ has rank $r+1$, $(V^*_n)^{[n]}$ has rank $n(r+1)$. Hence $c_{2n}((V^*_n)^{[n]})$ is expected to count the number of points where $n(r+1) -2n+1$ general sections drop rank. This number is equal to $\chi(V^*_n)$, and we expect it to also be equal to  $h^0(V^*_n)$. Hence the number of points where \eqref{eq:V*sec} drops rank is counted by $c_{2n}\left((V^*_n)^{[n]}\right)$.

From the preceding discussion, and encouraged by the numerical success of the finite Quot scheme method in \cite{bertram_potiers_2016}, we obtain (for $\hat E_n$ with chern character $e_n^\vee$):
\begin{conj} \label{conj:c2n=chi}
	\[
	c_{2n}\left((V^*_n)^{[n]}\right) = \chi(S^{[n]}, \det \hat E^{[n]}_n)
	\]
\end{conj}

\section{The main Conjecture} \label{sec:the-conj}

From Conjecture \ref{conj:c2n=chi}, one can derive the relationship between the universal series of Theorems \ref{thm:chi-series} and \ref{thm:c2n-series} stated in Conjecture \ref{conj:series}.
\begin{thm} \label{thm:conj-equiv}
Conjecture \ref{conj:c2n=chi} is equivalent to Conjecture \ref{conj:series}.
\end{thm}

Before embarking on the proof, we state a result that follows from Lagrange Reversion or the Lagrange-B\"{u}rmann formula. Recall that for a power series $f$, $[z^n]f(z)$ means ``the coefficients of $z^n$ in $f(z)$".
\begin{prop} \label{prop:lagrange}
	Let $\psi$ and $\phi$ be power series with constant term equal to 1.
	Suppose $f(z) = \sum_n \left([x^n]\;\psi(x)\phi(x)^n\right) z^n$. Let $w$ be defined implicitly by $w=z\phi(w)$. We may also view $z$ as a function of $w$. 
	Then \[f(z) = \frac{\psi(w)}{\phi(w) \frac{dz}{dw}}.\]
\end{prop}

\begin{proof}[Proof of Theorem \ref{thm:conj-equiv}]
Let $S$ be a del Pezzo surface with canonical class $K$. Using the notation of Section \ref{sec:length-taut} and setting $s = r+1$, we obtain from  Theorem \ref{thm:c2n-series} and equation \eqref{eq:c2V*} that 
\begin{align*}
  c_{2n}\left( (V^*_n)^{[n]}\right) &=  [x^n]\;V_s(x)^{\chi(L) - (n-1)(r-1)} \cdot W_s(x)^{\chi(L)}\cdot X_s(x)^{\frac12} \cdot Y_s(x)^{K.L - \frac{K^2}{2}} \cdot Z_s(x)^{K^2}   \\
 &= [x^n]\; (V_s(x)^{2-s})^{n-1}\cdot(W_s(x)V_s(x))^{\chi(L)}\cdot X_s(x)^{\frac12} \cdot Y_s(x)^{K.L - \frac{K^2}{2}} \cdot Z_s(x)^{K^2} 
\end{align*}
We now apply Proposition \ref{prop:lagrange} with $\phi(x) = V_s(x)^{2-s}$ and 
\[
\psi(x) = (W_s(x)V_s(x))^{\chi(L)}\cdot X_s(x)^{\frac12} \cdot Y_s(x)^{K.L - \frac{K^2}{2}} \cdot Z_s(x)^{K^2} \cdot \phi(x)^{-1}.
\]
The formula says that if $w = z\phi(w)$, we have
\begin{align}
  \sum_{n=0}^{\infty} & c_{2n} \left((V^*_n)^{[n]}\right)z^n = \frac{\psi(w)}{\phi(w) \frac{dz}{dw}} \\
  &= (W_s(w)V_s(w))^{\chi(L)}\cdot \left(\frac{X_s(w)}{\phi(w)^4\left(\frac{dz}{dw}\right)^2}\right)^{\frac12} \cdot Y_s(w)^{K.L - \frac{K^2}{2}} \cdot Z_s(w)^{K^2}. \label{eq:proof1}
\end{align}

Next,  since the determinant of $e_n^\vee$ is $L$ (independently of $n$) Theorem \ref{thm:chi-series} gives us:
\begin{equation}
 \sum_{n=1}^\infty \chi(\det \hat E_n^{[n]})z^n = g_r(z)^{\chi(L)} \cdot f_r(z)^{\frac{1}{2}} \cdot A_r(z)^{K.L-\frac{1}{2}K^2}\cdot B_r(z)^{K^2} \label{eq:proof2}
\end{equation}

Conjecture \ref{conj:c2n=chi} says that \eqref{eq:proof1} and \eqref{eq:proof2} are equal.
Since there are plenty of choices of $L$ and $S$ for which Conjecture \ref{conj:c2n=chi} is expected to hold, the equivalence follows.
\end{proof}

\begin{remark}
	We have verified Conjecture \ref{conj:series} computationally up to order 6 using Sage \cite{sage}.  The series $A_r(z)$ and $B_r(z)$ can be computed by the localization techniques in \cite{ellingsrud_homology_1987} and \cite{ellingsrud_botts_1996}, as suggested in \cite{ellingsrud_cobordism_2001}. We used up to order 6 in the present work (although our code can compute more, see \cite{johnson_two_2016}). We reproduce just the first few terms here:
\begin{align*}
A_r(z) &= 1 + \left(-\tfrac{1}{6} r^{3} + \tfrac{1}{6} r\right)z^{2} +
\left(\tfrac{17}{40} r^{5} - \tfrac{5}{8} r^{3} + \tfrac{1}{5}
r\right)z^{3} + \cdots \\
B_r(z) &= 1 + \left(-\tfrac{1}{24} r^{4} + \tfrac{1}{24} r^{2}\right)z^{2} +
\left(\tfrac{97}{720} r^{6} - \tfrac{31}{144} r^{4} + \tfrac{29}{360}
r^{2}\right)z^{3} + \cdots
\end{align*}

The $c_{2n}$ series, to our knowledge, have not been computed in general before. We computed them up to order 6, implementing an algorithm  based on Lehn's paper \cite{lehn_chern_1999}. We reproduce just the first few terms here:
\begin{align*}
V_s(w) &= 1 + w + \left(-\tfrac{1}{2} s^{2} + \tfrac{3}{2} s - 1\right)w^{2} + \left(\tfrac{1}{2} s^{4} - \tfrac{17}{6} s^{3} + 6 s^{2} - \tfrac{17}{3} s + 2\right)w^{3} + \cdots
\\
W_s(w) &= 1 + \left(-\tfrac{1}{2} s^{2} + \tfrac{3}{2} s - 1\right)w^{2} + \left(s^{4} - \tfrac{17}{3} s^{3} + 12 s^{2} - \tfrac{34}{3} s + 4\right)w^{3} + \cdots\\
X_s(w) &= 1 + \left(\tfrac{1}{2} s^{4} - 3 s^{3} + \tfrac{13}{2} s^{2} - 6 s + 2\right)w^{2} + \\
 &\left(-\tfrac{4}{3} s^{6} + 11 s^{5} - \tfrac{112}{3} s^{4} + 67 s^{3} - \tfrac{202}{3} s^{2} + 36 s - 8\right)w^{3} + \cdots\\
Y_s(w) &= 1 + \left(-\tfrac{1}{6} s^{3} + \tfrac{1}{2} s^{2} - \tfrac{1}{3} s\right)w^{2} + \left(\tfrac{17}{40} s^{5} - \tfrac{59}{24} s^{4} + \tfrac{127}{24} s^{3} - \tfrac{121}{24} s^{2} + \tfrac{107}{60} s\right)w^{3} + \cdots \\
Z_s(w) &= 1 + \left(-\tfrac{1}{24} s^{4} + \tfrac{1}{6} s^{3} - \tfrac{5}{24} s^{2} + \tfrac{1}{12} s\right)w^{2} + \\
&\left(\tfrac{97}{720} s^{6} - \tfrac{107}{120} s^{5} + \tfrac{83}{36} s^{4} - \tfrac{35}{12} s^{3} + \tfrac{1303}{720} s^{2} - \tfrac{53}{120} s\right)w^{3} + \cdots
\end{align*}
\end{remark}
\begin{remark} \label{rem:w}
	The first few terms of the expansion of $w$ in terms of $z$ are 
	\begin{align*}
w(z) &= z + \left(- r + 1\right)z^{2} + \left(\tfrac{1}{2} r^{3} + \tfrac{1}{2} r^{2} - 2 r + 1\right)z^{3} + \\
&\left(-\tfrac{1}{2} r^{5} - \tfrac{2}{3} r^{4} + \tfrac{3}{2} r^{3} + \tfrac{5}{3} r^{2} - 3 r + 1\right)z^{4} + \\ &\left(\tfrac{2}{3} r^{7} + \tfrac{23}{24} r^{6} - \tfrac{11}{6} r^{5} - \tfrac{73}{24} r^{4} + \tfrac{8}{3} r^{3} + \tfrac{43}{12} r^{2} - 4 r + 1\right)z^{5} + \cdots
	\end{align*}
\end{remark}
\begin{remark}
	Although Conjecture \ref{conj:c2n=chi} is expected only for del Pezzo surfaces, Conjecture \ref{conj:series} is expected to hold for all surfaces. This is not surprising, since, for example, knowledge of all the numbers for $\PP^2$ and $\PP^1 \times \PP^1$ could determine the series in Theorem \ref{thm:chi-series} and Theorem \ref{thm:c2n-series}.
\end{remark}
\begin{remark}
	In \cite{ellingsrud_cobordism_2001}, the authors note that Serre duality gives the symmetries $B_{-r}=B_r$, $f_r = f_{-r}$, $g_r=g_{-r}$, and $A_r = 1/A_{-r}$. These translate via Conjecture \ref{conj:series} into somewhat more mysterious symmetries such as $Z_s(w)=Z_{2-s}(v)$ and $Y_s(w)=1/Y_{2-s}(v)$, where $v$ and $w$ are related by $vV_s(w)^{2-s}=wV_{2-s}(v)^s$.
\end{remark}

\section{Small ranks} \label{sec:small-rank}
We analyze Conjecture \ref{conj:series} for some small values of $r$ and $s$.

\subsection{Case: $s=2,r=1$} \label{sec:s=2}
Let $F$ be a sufficiently positive a rank 2 bundle. Then $F^{[n]}$ is a rank $2n$ bundle, and $c_{2n}$ computes the number of points of vanishing of a general section. The section $s^{[n]} \in H^0(F^{[n]},S^{[n]})$ vanishes at points $[Z]$ such that $Z$ is contained in the vanishing of $s\in H^0(F,S)$. Hence we conclude that
\[
 c_{2n}(F^{[n]}) = \binom{c_2(F)}{n}.
\]
Since there are plenty of such rank $2$ bundles, it follows then that $W_2=X_2=Y_2=Z_2=1$ and $V_2(w) = 1+w$. We know that $A_1=B_1=f_1=1$ and $g_1(z)=1+z$, and $\phi_2 = 1$ so $z=w$. This verifies Conjecture \ref{conj:series} completely in this case.

\subsection{Case: $s=1,r=0$}
When $F$ is a line bundle, $F^{[n]}$ is a bundle of rank $n$. Hence $c_{2n}(F)=0$ (for $n>0$). There are plenty of line bundles and surfaces, so we see that $W_1=X_1=Y_1=Z_1=1$ We also know that the series $A_0,B_0,f_0$ are 1, so \eqref{C3}-\eqref{C4} are verified.

Plugging in $r=0$ into our computation of $w$ (see Remark \ref{rem:w}), one can guess that $w=\frac{z}{1-z}$. Assuming this, we can verify \eqref{C1}-\eqref{C2}:

We know that $g_0(z) = \frac{1}{1-z}$. The equation $w = z\phi_1(w)$ implies that $\phi_1(w) = \frac{1}{1-z}$. Since $V_1(w)=\phi_1(w)$, this confirms \eqref{C1}.

One can check with implicit differentiation that $\frac{dz}{dw} = (1-z)^2$. This confirms \eqref{C2}.

\subsection{Case: $s=0,r=-1$} \label{sec:s=0}
We know that $A_{-1} = B_{-1} = f_{-1} = 1$ and $g_{-1}(z) = 1+z$, so \eqref{C3} and \eqref{C4} predict that $Y_0=Z_0=1$. Letting $F=0$ on various surfaces, we can at least see that $Z_0^2/Y_0 = 1$.

Letting $F=0$ and $S$ be $K$-trivial surface, we see that $W_0^2X_0 = 1$. 
 Again plugging $r=-1$ into our computation for $w$ (see Remark \ref{rem:w}), we can guess that $w= \frac{z}{(1-z)^2}$.

We have one fact and three guesses to determine the three series $V_0$, $W_0$, and $X_0$, so we can at least check that they are consistent. Let's assume that $w=\frac{z}{(1-z)^2}$ and \eqref{C1}, and check that this implies \eqref{C2}.

We get $\phi_0(w) = \frac{1}{(1-z)^2}$, hence $V_0(w) = \frac{1}{1-z}$. Then \eqref{C1} implies that $W_0(w) = (1-z)(1+z)$. Hence $X_0(w)$ must be $\frac{1}{(1-z)^2(1+z)^2}$.  By implicit differentiation, we see that $\frac{dz}{dw} = \frac{(1-z)^3}{1+z}$, and plugging all these in, we see that the right hand side of \eqref{C2} is $1$, as desired.

In particular, we obtain a prediction about skyscraper sheaves on any surface:
\[
\sum_n  \ c_{2n}\left(\mathcal O_p^{[n]}\right)w^n = 1-z
\]

We can get a more explicit formula for this, as we explain now. Let $C_n$ be the $n$th Catalan number, which is given by the formula $C_n = \frac{1}{n+1}\binom{2n}{n}$. Let $C(w) = \sum_{n\ge0} C_n w^n$ be the ordinary generating series. Recall that the Catalan numbers are determined by $C_0 = 1$ and the functional equation $C(w)=1+wC(w)^2$. It follows then that the series for the alternating Catalan numbers $\hat C(w) = \sum_{n\ge0}(-1)^nC_nw^n$ is determined by $C_0=1$ and the equation
\[
\hat C(w) = 1-w\hat C(w)^2
\]
Substituting $w= \frac{z}{(1-z)^2}$, one obtains
\[
\hat C(w) = 1-\frac{z}{(1-z)^2}\hat C(w)^2.
\] 
This equation has the solution $\hat C(w) = 1-z$. Hence we are led to predict that 
\[
c_{2n}\left(\mathcal O_p^{[n]}\right) = (-1)^n C_n
\]
We will prove this later as Theorem \ref{thm:cat}.

\subsection{Case: $s=-1,r=-2$} \label{sec:s=-1}
First, we observe that, making the change of variables $z = u(1+u)^{r^2-1}$, we have
\begin{align*}
g_r(z) = 1+u \\
f_r(z) = \frac{(1+u)^{r^2}}{1 + r^2u}
\end{align*}
To see this, one can write 
\begin{align*}
g_r'(z) = \sum_k\left( [z^k] (1+z)^{(1-r^2)(k+1)}\right) z^k \\
f_r(z) = \sum_k\left( [z^k](1+z)^{(1-r^2)(k-1)} \right)z^k
\end{align*}
Then apply Proposition \ref{prop:lagrange} with $\phi(z) = (1+z)^{1-r^2}$ and we obtain the formula for $f_r(z)$ above and $g_r'(z)= \frac{1}{\frac{dz}{du}}$, from which the formula for $g_r(z)$ follows. 

Specializing to $r=-2$ we have $z = u(1+u)^3$, $g_{-2}(z) = 1+u$, and $f_{-2}(z) = \frac{(1+u)^4}{1+4u}$. We can't say anything about $A_{-2}$ or $B_{-2}$ (beyond their computation for small $n$), so we will restrict ourselves to the $K$-trivial case.

In \cite{marian_segre_2015}, Marian, Oprea, and Pandharipande compute the top Segre class of the tautological sheaf associated to a line bundle on a $K$-trivial surface, verifying the $K$-trivial case of a conjecture of Lehn \cite{lehn_chern_1999}. From this we get some information about $s=-1$. 

Let $F = -\mathcal O(H)$, so (on a $K$-trivial surface) we have $\chi(c_1(F)) = \frac12 H^2 + \chi(S)$, $c_2(F) = H^2$, and $c_2(S) = 12\chi(S)$.
It then follows from formulas (1),(3), and (4) in \cite{marian_segre_2015} that, using the substitution $w = \frac12t(1+t)^2$, we have
\begin{align*}
(1+t)^{\frac12} =&  V_{-1}(w) W_{-1}(w)^{\frac12} \\ 
(1+t)^{\frac18}(1+3t)^{-\frac{1}{24}} =&   W_{-1}(w)^{\frac{1}{12}}X_{-1}(w)^{\frac{1}{24}} 
\end{align*}
We thus obtain 
\begin{align*}
(1+t) &= V_{-1}(w)^2W_{-1}(w) \\
\frac{(1+t)^3}{(1+3t)} &= W_{-1}(w)^2 X_{-1}(w)
\end{align*}
 If we assume that \eqref{C1} is true, we get 
\[
g_{-2}(z) = W_{-1}(w)V_{-1}(w)
\]
so we obtain ``closed'' formulas:
\begin{align*}
V_{-1}(w) &= \frac{1+t}{1+u} \\
W_{-1}(w) &= \frac{(1+u)^2}{1+t} \\
X_{-1}(w) &= \frac{(1+t)^5}{(1+3t)(1+u)^4}.
\end{align*}

From \eqref{C1}, let us verify \eqref{C2}. We can do this by writing both sides of \eqref{C2} in terms of $t$.
We have $\phi_{-1}(w) = V_{-1}(w)^3 = \frac{(1+t)^3}{(1+u)^3}$, and from the equations $w = z\phi_{-1}(w)$ and $z=u(1+u)^3$ we can obtain $u = \frac{t}{2(1+t)}$. Thus we can write $\phi_{-1}(w)$,  $X_{-1}(w)$, and $f_{-2}(z)$ in terms of $t$. In order to obtain $\frac{dz}{dw}$, we first write $z = u(1+u)^3 = \frac{(2+3t)^3}{2+2t)^4}$. We know that
$
\frac{dz}{dw} = \frac{dz}{dt} \frac{dt}{dw}
$.
One can obtain $\frac{dt}{dw} = \left[\frac12(1+t)^2 + t(1+t)\right]^{-1}$ by implicit differentiation applied to $w = \frac12t(1+t)^2$. Now one plugs all these in and, preferably with a computer, verifies \eqref{C2}.

\section{Top Chern classes of tautological sheaves associated to skyscraper sheaves}
In this section will prove the following, suggested by our analysis in Section \ref{sec:s=0}.

\begin{thm} \label{thm:cat}
	We have on ${S^{[n]}}$:
	\[
	 c_{2n}\left(\mathcal O_p^{[n]}\right) = (-1)^n C_n
	\]
	for any surface $S$ and any point $p$, where $C_n$ is the $n$th Catalan number.
\end{thm}

\subsection{Preliminaries}
We will prove Theorem \ref{thm:cat} using the methods of Lehn \cite{lehn_chern_1999}. We recall some notation and important results from that paper.

Let $\HH_n = H^*(S^{[n]},\mathbb Q)$, and let $\HH = \bigoplus_n \HH_n$. We consider $S^{[0]}$ to be a point, so $\HH_0$ is spanned by a single vector $\one$ called the vacuum vector. $\HH$ has two gradings: the conformal degree given by the decomposition $\bigoplus_n \HH_n$ and the algebraic degree, given by the algebraic codimension of the cohomology class. (The algebraic codimension is half of the cohomological codimension.)

For each $n \in \mathbb Z$ and $\alpha \in H^*(S,\QQ)$ there is an  operator $\q_n(\alpha)$ on $\mathbb H$ of conformal degree $n$ and algebraic degree $n-1+\deg(\alpha)$. When $n$ is positive and $x\in \HH_k$ is represented by a subscheme $X \subset S^{[k]}$ and $\alpha$ is represented by a subscheme $A \subset S$, then $\q_n(\alpha)(x)$ roughly corresponds to the class in $H^*(S^{[n+k]},\QQ)$ formed by taking the locus of subschemes where $k$ of the points correspond to a point from $X$ and the remaining $n$ points lie on $A$ (with multiplicity if this can happen in more than one way). We also have $\q_0(\alpha) = 0$. 
\begin{example} \label{ex:q}
Letting $p$ be the class of a point on $S$ and writing $S$ for the fundamental class, we have that $\q_1(p)^n\one$ is the class of a point on $S^{[n]}$ and $\q_1(S)^n\one$ is $n!$ times the fundamental class of $S^{[n]}$.
\end{example}
We refer the reader to \cite{lehn_chern_1999} for the precise definition and more details. For this paper, we will only need the  properties of the $\q$ operators that we recall here.

\begin{thm}[\cite{nakajima_heisenberg_1997},\cite{grojnowski_instantons_1996}] \label{thm:q-comm}
	\[
	[\q_n(\alpha),\q_m(\beta)] = n \delta_{n+m} \int_S \alpha \beta \operatorname{id}_{\mathbb H}
	\]
\end{thm}
Here $[\mathfrak f, \mathfrak g] := \mathfrak {f g} - \mathfrak{gf}$ is the commutator and $\delta_i=1$ if $i=0$ and $0$ otherwise.

Let $\mathfrak d$ be the operator on $\HH$ that multiplies $x \in \HH_n$ by $c_1(\OO_S^{[n]})$. We define the derivative of $\q_n(\alpha)$ to be 
\[
\q_n'(\alpha) := [\mathfrak d, \q_n(\alpha)].
\]
This derivative obeys the Leibniz rule. We write $\q_n^{(k)}(\alpha)$ for the $k$th derivative, not to be confused with $\q_n(\alpha)^k$, the $k$-fold composition.

Derivatives may be computed in terms of $\q$'s.  Let $\delta:H^*(S,\QQ) \rightarrow H^*(S,\QQ) \otimes H^*(S, \QQ)$ be the map induced by the diagonal embedding, and write $\delta(\alpha) = \sum_i \beta_i \otimes \gamma_i$.  Define 
\[
\q_n\q_m \delta (\alpha) := \sum_i \q_n(\beta_i)\q_m(\gamma_i). 
\] 
Let $K$ be the canonical class of $S$. Then we have:
\begin{thm}[\cite{lehn_chern_1999}] \label{thm:q-deriv}
	\[
	\q_n'(\alpha) = \frac n2 \sum_\nu \q_\nu \q_{n-\nu} \delta(\alpha) + \binom{n}{2}\q_n(K_S.\alpha)
	\]
\end{thm}
Notice that although the sum is over all $\nu\in \mathbb Z$, the operator is locally finite, that is, for any fixed $k$, only finitely many of the terms have a non-zero action on $\HH_k$.

To compute Chern classes of tautological sheaves, one needs: 
\begin{thm}[\cite{lehn_chern_1999}, Corollary 4.3] \label{thm:lehn}
	For any $u \in K(S)$ let $\mathfrak C(u)$ be the operator
	\[
	\mathfrak C(u) := \sum_{\nu,k \ge 0} \binom{\text{rk}(u)-k}{\nu} \q_1^{(\nu)}(c_k(u)).
	\]
	Then
	\[
	\sum_{n\ge 0} c(u^{[n]}) = \exp(\mathfrak C(u)) \one.
	\]
\end{thm}
For the remainder of the paper, let us put $u = \OO_p$, so $c_0(u)$ is the fundamental class of $S$, which, by a slight abuse we denote by $S$, $c_1(u)=0$, and $-c_2(u)$  is the class of a point, which we denote by $p$. Plugging this in, we obtain
\begin{equation}
\mathfrak C(u) = \q_1(S) + \sum_{\nu \ge 0} (-1)^{(\nu+1)} (\nu+1) \q_1^{(\nu)}(p) \label{eq:C}
\end{equation}
We see that we will be interested in the following special case of Theorem \ref{thm:q-deriv}. Since $\delta(p) = p \otimes p$,  have:
\begin{cor} \label{cor:qp-deriv}
	\[
	\q_n'(p) = \frac n2 \sum_{\nu} \q_\nu(p) \q_{n-\nu}(p)
	\]
\end{cor}

To prove Theorem \ref{thm:cat}, we must show that the series
\begin{equation}
H(w) := \sum_{n=0}^\infty (-1)^n w^n \frac{1}{n!} \int_{S^{[n]}} \mathfrak C(u)^n \one \label{eq:H}
\end{equation}
is the generating function of the Catalan numbers.

The next proposition collects some standard facts.
\begin{prop} \label{lem:3ints}
	
	(1) Suppose $i_1, \dots, i_k$ are non-negative integers. Then the integral
	\begin{equation*}
	\int_{S^{[n]}} \q_{i_1}(\alpha_1) \cdots \q_{i_k}(\alpha_k) \one \label{eq:gen-int}
	\end{equation*}
	is equal to 0 unless $i_1= \dots = i_k = 1$ and $k=n$.

	(2) Suppose $\mathfrak m$ is a monomial in $\q_1(S)$ and $\q_{k_i}(p)$ for various integers $k_i$. Then $\int_{S^{[n]}} \mathfrak m \one = 0$ unless each $k_i$ is equal to $1$ or $-1$.
	
	(3)
	$
	\int_{S^{[n]}} \q_1(p)^n \one=1.
	$
\end{prop}
\begin{proof}
	Part (1) follows by comparing the algebraic degree to the dimension of $S^{[n]}$.
		
	For (2), notice that if any $k_j$ was less than $-1$, it would commute with all other factors, and so could be moved to the right and act on $\one$ to get 0. (Note that $\q_k(\alpha)\one=0$ for any $k\le 0$.) If any $k_j$ was greater than one, then we use the commutation formula to move any $\q_{-1}$'s towards the right without disturbing the large $k_j$. Then  apply (1).
	 
	Part (3) follows from Example \ref{ex:q}.
\end{proof}

\begin{lem} \label{lem:biject}
	Suppose $\mathfrak m$ is a monomial of conformal degree $n$ containing factors of $\q_1(p)$, $\q_{-1}(p)$, and $\q_1(S)$. Then
	$\int_{S^{[n]}} \mathfrak m \one$ is equal to the number bijections between the $\q_{-1}(p)$'s and the $\q_1(S)$'s so that each $\q_1(S)$ is to the right of its associated $\q_{-1}(p)$, times a factor of $-1$ for each $\q_{-1}(p)$.
	
	In particular, if the number of $\q_{-1}(p)$'s is not equal to the number of $\q_{1}(S)$'s, then   $\int_{S^{[n]}} \mathfrak m \one=0$.
\end{lem}
\begin{proof}
	Letting $K_\ell$ and $P_\ell$ denote the partial sums $\sum_{j=\ell}^m k_j$ and $\sum_{j=\ell}^m p_j$ respectively,  the lemma is equivalent to the assertion that
	\begin{align}
	\int_{S^{[n]}} \q_1(p)^n &\prod_{j=1}^m \left(\q_{-1}(p)^{p_j} \q_1(S)^{k_j} \right)\one = \\
	&\prod_{\ell=1}^m (-1)^{p_\ell}(K_\ell - P_{\ell+1})(K_\ell - P_{\ell+1}-1) \cdots (K_\ell - P_{\ell}+1) \label{eq:KP}
	\end{align}
	By Theorem \ref{thm:q-comm}, we have the recurrence
	\begin{align*}
	&\int_{S^{[n]}} \q_1(p)^n \prod_{j=1}^m \left(\q_{-1}(p)^{p_j} \q_1(S)^{k_j} \right)\one \\
	= &\int_{S^{[n]}} \q_1(p)^n \prod_{j=1}^{m-1} \left(\q_{-1}(p)^{p_j} \q_1(S)^{k_j} \right) \left(\q_{-1}(p)^{p_m-1} \q_1(S)\q_{-1}(p)\q_1(S)^{k_m-1}\right)\one\\
	- &\int_{S^{[n]}} \q_1(p)^n \prod_{j=1}^{m-1} \left(\q_{-1}(p)^{p_j} \q_1(S)^{k_j} \right)  \left(\q_{-1}(p)^{p_m-1}\q_1(S)^{k_m-1}\right) \one
	\end{align*}
	To prove the formula \eqref{eq:KP} by induction, we substitute it into the recurrence and cancel the first $m-1$ factors which appear identically on both sides. We are then reduced to proving
	\begin{align*}
	&(-1)^{p_m}(k_m)(k_m-1) \cdots (k_m-p_m+1) \\
	= &(-1)(k_m-1) \cdot (-1)^{p_m-1}(k_m - 1)(k_m - 2) \cdots (k_m-p_m+1) \\
	-&(-1)^{p_m-1}(k_m - 1)(k_m - 2) \cdots (k_m-p_m+1)
	\end{align*}
	which is clear.
\end{proof}
 
\begin{lem} \label{lem:no-odd}
Suppose $\mathfrak m$ is a monomial in $\q_1(X)$ and derivatives of $\q_1(p)$ and
\begin{equation} 
 \int_{S^{[n]}} \mathfrak m \one \neq 0 \label{eq:intm}
\end{equation}
Then there are no derivatives of odd order occuring in $\mathfrak m$.
\end{lem}
\begin{proof}
Assume $\mathfrak m$ contains a factor $\q_1^{(s)}(p)$ with $s$ odd.
By iterating the formula in Corollary \ref{cor:qp-deriv}, we can reduce $\q_1^{(s)}(p)$ to a sum of monomials of the form $\prod_{i=1}^{s+1} \q_{n_i}(p)$, where $\sum_i {n_i}=1$. Since $s$ is odd, at least one of the $n_i$ must be even. We then similarly reduce all other derivatives, and apply Proposition \ref{lem:3ints}, (2).
\end{proof}

By Lemma \ref{lem:no-odd}, we see that we may omit the terms with $\nu$ odd in \eqref{eq:C}. Doing this, and letting $\rr_n = -(2n-1)\q_1^{(2n-2)}(p)$, we obtain
\[
\CCC(u) := \q_1(S) + \sum_{n\ge 1} \rr_n
\]
Iterating Corollary \ref{cor:qp-deriv}, we see that there are constants $D_n$ (which we will compute later) so that
\begin{equation}
\rr_n = -nD_n \q_1(p)^{n}\q_{-1}(p)^{n-1} + R_n \label{eq:D_n}
\end{equation}
 where $R_n$ is composed of terms with a $\q_k(p)$ with $k \neq \pm 1$ (which will not contribute, by Lemma \ref{lem:3ints} (2)).

The operator $\rr_n$ in fact has conformal degree $1$, not $n$. However we have chosen the subscript since, by Lemma \ref{lem:biject}, each $\rr_n$ in a term from the expansion of $\exp(\CCC(u))$ must be accompanied by $(n-1)$ $\q_1(S)$'s (or the term will integrate to $0$), and altogether this gives conformal degree $n$.

Now let $\bar y = (y_1, y_2, \dots)$ be a sequence of commuting variables indexed by positive integers and we put
\[
\CCC(u)(\bar  y) := \q_1(S) + \sum_{n\ge 1} y_n\rr_n.
\]
It follows from our discussion above that 
\begin{equation}
H(w)=\sum_n \frac{(-1)^n}{n!}\int_{S^{[n]}} \CCC(u)(w,w^2,w^3,\dots)^n\one . \label{eq:Hww}
\end{equation}

Assume $n_1, \dots, n_k$ are positive integers with $\sum n_i=n$. Let $\Aut(n_1, \dots, n_k)$ be the product of the factorials of the multiplicities with which each integer appears in the list $n_1, \dots, n_k$. Then we have:
\begin{lem} \label{lem:coef-Aut}
 The coefficient of $y_{n_1} \cdots y_{n_k}$ in $\frac{(-1)^n}{n!}\int_{S^{[n]}}\CCC(u)(\bar y)^n\one$ is \[\frac{1}{\Aut(n_1, \dots, n_k)}\prod_{j=1}^k D_{n_j}.\]
\end{lem}
\begin{proof}
 Consider a monomial $\mathfrak m$ in the $\rr_{n_i}$'s and $n-k$ $\q_1(S)$'s. To evaluate $\int_{S^{[n]}} \mathfrak m \one$, we use Lemma \ref{lem:biject}. Let $N_{\mathfrak m}$ be the number of ways to assign $n_i-1$ $\q_1(S)$'s to each $\rr_{n_i}$ so that each $\q_1(S)$ is assigned and is to the right of its assigned $\rr_{n_i}$. After substituting in \eqref{eq:D_n}, we can extend this assignment to a bijection between $\q_1(S)$'s and $\q_{-1}(p)$'s in $\prod (n_i-1)!$ ways. So we obtain
 \begin{align}
 \int_{S^{[n]}} \mathfrak m \one &= N_{\mathfrak m} \left(\prod_{j=1}^k (-n_j D_{n_j})\right) \left(\prod_{j=1}^k (-1)^{n_j-1}(n_j-1)!\right)\\
  &= (-1)^n N_{\mathfrak m} \prod_{j=1}^k n_j!D_{n_j}. \label{eq:N_m}
 \end{align}

Next, we observe that
\[
\sum_{\mathfrak m} N_{\mathfrak m} = \frac{1}{\Aut(n_1, \dots, n_k)}\binom{n}{n_1, \dots, n_k}
\]
where the sum is over all monomials $\mathfrak m$ in the $\rr_{n_i}$'s and $n-k$ $\q_1(S)$'s. To see this, start with $n$ empty spots. Then, for each $i$, pick $n_i$ of the spots, and insert $\rr_{n_i}$ is the left most of the selected spots, and $n_i-1$ $\q_1(S)$'s in the remaining selected spots. One can see that his process will construct each $\mathfrak m$ together with an assignment (as counted by $N_{\mathfrak m}$), overcounted by $\Aut(n_1, \dots, n_k)$.

We now sum \eqref{eq:N_m} over all $\mathfrak m$ and multiply by $\frac{(-1)^n}{n!}$ to obtain the Lemma.
\end{proof}

\begin{proof}[Proof of Theorem \ref{thm:cat}]
The coefficients in Lemma \ref{lem:coef-Aut} are exactly those that show that
\[\sum_n \frac{(-1)^n}{n!}\int_{S^{[n]}} \CCC(u)(\bar y)^n\one = \exp\left(\sum_{n\ge 1} D_{n}y_n\right)
\]
Hence, by \eqref{eq:Hww}, we have 
	 $H(w) = \exp\left(\sum_{n\ge 1} D_nw^n\right)$.
Thus we have reduced the proof of Theorem \ref{thm:cat} to the claim that the generating series of the numbers $D_n$ is the logarithm of the generating series of the Catalan numbers.

Differentiating the formula $C(x) = 1 + xC(x)^2$ and then dividing by $C(x)$, one obtains
\[
C'(x)/C(x) = C(x) + 2xC'(x).
\]
Integrating this, one obtains that the coefficient of $x^n$ in $\log(C(x))$ is $\frac1n C_{n-1} + 2\frac{n-1}{n} C_{n-1} = \frac{2n-1}{n}C_{n-1}= \frac{2n-1}{n}\cdot\frac{(2n-2)!}{n!(n-1)!}$. Lemma \ref{lem:Dn} in the next section will check that this number is equal to $D_n$.
\end{proof}

\subsection{Trees}
A \emph{tree} is a connected graph with no cycles.

A \emph{binary tree} is a tree with a distinguished node called the root. Furthermore, each node has either no children, a left child, a right child, or both. A binary tree is \emph{full} if each node has 0 or 2 children. An \emph{increasing binary tree} is a binary tree that comes equipped with a (total) ordering of the nodes so that a parent always precedes its children. A \emph{leaf} is a node with no children. An \emph{internal node} is a node with at least one child. Let $B_n$ be the set of all increasing binary trees with $n$ nodes. We call the elements of $B_n$ increasing binary $n$-trees. We consider the empty tree to be the unique increasing binary $0$-tree.

Given an increasing binary $n$-tree $T$, one can produce a full binary tree $\tilde T$ by adding $n+1$ leaves in the obvious way. Let us associate variables $x_1, \dots, x_{n+1}$ to each of these new leaves, from left to right. Then each internal node has a ``weighted hook length" given by the sum of the variables associated to the leaves descended from that node. Then we define the weight $\tilde T_{\wt}(x_1, \dots, x_n)$ of $\tilde T$  to be the product of the weighted hook lengths over all internal nodes. We define a polynomial $\tilde T(x_1, \dots, x_{n+1})$ to be the sum of the weights of $\tilde T$ over all possible permutations of $x_1, \dots, x_{n+1}$. Finally, we define
\[
F_n(x_1, \dots, x_{n+1}) = \sum_{T \in B_n} \tilde T(x_1, \dots, x_{n+1}).
\]

Our interest in these trees comes from the following observation.
\begin{lem} \label{lem:dp}
	\[
	\q^{(n)}_k(p) = \frac{1}{2^n}\sum_{\vi \text{ sorted}}  \frac{F_n(\vi)}{\Aut(\vi)}\q_{i_1}(p) \cdots \q_{i_{n+1}}(p)
	\]
	where the sum is over all vectors $\vi = (i_1, \dots, i_{n+1})$ of integers with $i_{1} \ge \dots\ge i_{n+1}$ and $\sum i_j = k$. 
\end{lem} 
\begin{proof}
After iterating Corollary \ref{cor:qp-deriv} and using the product rule, one can see that
\[
 \q^{(n)}_k(p) = \frac{1}{2^n}\sum_{T\in B_n}\sum_{\vi {\text{ unsorted}}}  \tilde T_{\wt}(\vi)\; \q_{i_1}(p) \cdots \q_{i_{n+1}}(p)
\]
where the second sum is over \emph{all} vectors $\vi$ with $\sum i_j = k$.

Since the $\q_{i_j}(p)$ commute with each other, we may sort them in each term and obtain
\[
\q^{(n)}_k(p) = \frac{1}{2^n}\sum_{T\in B_n}\sum_{\vi {\text{ sorted}}}  \frac{\tilde T(\vi)}{\Aut(\vi)}\; \q_{i_1}(p) \cdots \q_{i_{n+1}}(p)
\]
The $\Aut$ factor compensates for the fact that $\tilde T(\vi)$ is the sum over all permutations, including those that leave $\vi$ invariant.

But now the result follows from the definition of $F_n$.
\end{proof}

We thank Gjergji Zaimi, who pointed out to us the following lemma and its proof.
\begin{lem} \label{lem:F_n}
We have
\begin{equation}
F_n(x_1, \dots, x_{n+1}) = 2^nn!(x_1 + \cdots + x_n)^n \label{eq:Fn}
\end{equation}
\end{lem}
\begin{proof}
A binary tree can be split into two (possibly empty) binary trees by deleting the root. An increasing binary $k$-tree and an increasing binary $m$-tree can be joined into a binary $(k+m+1)$-tree by attaching both roots to a new root. The ordering can be extended in $\binom{k+m}{k}$ ways to form an increasing binary $(k+m+1)$-tree. We also note that the root of a binary $n$-tree has hook length $(x_1 + \cdots x_{n+1})$.

Let $\mathcal P$ be the set of proper, non-empty subsets of $\{x_1, \dots, x_{n+1}\}$. The discussion of the previous paragraph gives us the recursion
\begin{equation}
	F_n(x_1, \dots, x_{n+1}) = (x_1 + \cdots x_{n+1})\sum_{S \in \mathcal P} \binom{n-1}{|S|-1}F_{|S|-1}(S)F_{n-|S|}(S^c). \label{eq:recur}
\end{equation}
This will allow us to prove our theorem inductively. The base case is trivial. Plugging in our proposed formula \eqref{eq:Fn}, we have
\begin{align*}
&\sum_{S \in \mathcal P} \binom{n-1}{|S|-1}F_{|S|-1}(S)F_{n-|S|}(S^c)  \\
=&\sum_{S \in \mathcal P'}\left[ \binom{n-1}{|S|-1} (|S|-1)!2^{|S|-1} \left(\sum_{S} x_i \right)^{|S|-1} 2^{n-|S|}(n-|S|)!\left( \sum_{S^c} x_i \right)^{n-|S|} \right]\\
=& (n-1)! 2^{n-1}  \sum_{S \in \mathcal P'} \left[\left(\sum_{S} x_i \right)^{|S|-1} \left( \sum_{S^c} x_i \right)^{n-|S|}\right].
\end{align*}
Substituting this and the formula for $F_n$ into \eqref{eq:recur}, we see that it suffices to show that 
\begin{equation}
2n(x_1 + \cdots x_{n+1})^{n-1} =  \sum_{S \in \mathcal P'}\left[\left(\sum_{S} x_i \right)^{|S|-1} \left( \sum_{S^c} x_i \right)^{n-|S|}\right] \label{eq:suf}
\end{equation}
This equation will follow from the weighted Cayley formula for labeled trees, which we recall now. Given a labeled set of nodes $S = \{1, \dots, p\}$, the formula says that: 
\[
 y_1 \cdots y_p(y_1 + \cdots + y_p)^{p-2} = \sum_T y_1^{\deg_T(1)}\cdots y_p^{\deg_T(p)}.
\]
Here, the sum runs over all (not binary!) trees $T$ on $S$, and $\deg_T(i)$ is the number of edges touching node $i$. If instead we sum over all rooted trees, where the root is considered to have a half edge which increases its degree by 1, the left hand side of the formula will be $y_1 \cdots y_p(y_1 + \cdots + y_p)^{p-1}$.

Now, a tree on $\{1, \dots, n+1\}$ can be constructed by taking a rooted tree on $S \in \mathcal P$ and a rooted tree on $S^c$ and joining the roots by gluing together their half edges. A given tree on $\{1, \dots, n+1\}$ will be constructed $2n$ times this way: there are $n$ edges to split, and $2$ choices for which side to be $S$ and which to be $S^c$. Since the degrees of the nodes are unchanged by the gluing, we obtain from the Cayley formula
\begin{align*}
\sum_{S \in \mathcal P} \left[\prod_{S} x_i \left(\sum_{S} x_i \right)^{|S|-1}  \prod_{ S^c} x_i \left(\sum_{ S^c} x_i \right)^{n-|S|} \right] = 2nx_1 \cdots x_{n+1}&(x_1 + \cdots x_{n+1})^{n-1}. 
\end{align*}
Dividing both sides by $x_1 \cdots x_{n+1}$, one obtains \eqref{eq:suf}.
\end{proof}

\begin{lem} \label{lem:Dn}
	We have $D_n = \frac{2n-1}{n} \cdot  \frac{(2n-1)!}{n!(n-1)!}$.
\end{lem}
\begin{proof}
	The number $D_n$ is determined by (see \eqref{eq:D_n}):
	\[
	-(2n-1)\q_1^{(2n-2)}(p) = -nD_n\q_1(p)^n\q_{-1}(p)^{n-1} + R_n .
	\]
	By Lemmas \ref{lem:dp} and \ref{lem:F_n} the left hand side is equal to \[-(2n-1)\frac{(2n-1)!}{n!(n-1)!}\q_1^n(p)\q_{-1}^{n-1}(p) + R_n,\] so we obtain the desired result.
\end{proof}

\bibliographystyle{alpha}

\bibliography{bib,extrabib}{}

\newcommand{\etalchar}[1]{$^{#1}$}
\begin{thebibliography}{MOP17}

\bibitem[BGJ16]{bertram_potiers_2016}
Aaron Bertram, Thomas Goller, and Drew Johnson.
\newblock Le {Potier}'s strange duality, quot schemes, and multiple point
  formulas for del {Pezzo} surfaces.
\newblock October 2016.
\newblock arXiv: 1610.04185.

\bibitem[Dan02]{danila_resultats_2002}
Gentiana Danila.
\newblock Résultats sur la conjecture de dualité étrange sur le plan
  projectif.
\newblock {\em Bulletin de la Société mathématique de France}, 130(1):1--33,
  2002.

\bibitem[EGL01]{ellingsrud_cobordism_2001}
Geir Ellingsrud, Lothar Göttsche, and Manfred Lehn.
\newblock On the cobordism class of the {Hilbert} scheme of a surface.
\newblock {\em Journal of Algebraic Geometry}, 10(1):81--100, 2001.

\bibitem[ES87]{ellingsrud_homology_1987}
Geir Ellingsrud and Stein~Arild Strømme.
\newblock On the homology of the {Hilbert} scheme of points in the plane.
\newblock {\em Inventiones mathematicae}, 87(2):343--352, June 1987.

\bibitem[ES96]{ellingsrud_botts_1996}
Geir Ellingsrud and Stein Strømme.
\newblock Bott’s formula and enumerative geometry.
\newblock {\em Journal of the American Mathematical Society}, 9(1):175--193,
  1996.

\bibitem[Gro96]{grojnowski_instantons_1996}
I.~Grojnowski.
\newblock Instantons and affine algebras {I}: {The} {Hilbert} scheme and vertex
  operators.
\newblock {\em Mathematical Research Letters}, 3(2):275--291, March 1996.

\bibitem[Joh16]{johnson_two_2016}
Drew Johnson.
\newblock {\em Two enumerative problems in algebraic geometry}.
\newblock PhD thesis, University of Utah, 2016.

\bibitem[Leh99]{lehn_chern_1999}
Manfred Lehn.
\newblock Chern classes of tautological sheaves on {Hilbert} schemes of points
  on surfaces.
\newblock {\em Inventiones mathematicae}, 136(1):157--207, March 1999.

\bibitem[LP05]{le_potier_dualite_2005}
J.~Le~Potier.
\newblock Dualité étrange sur le plan projectif.
\newblock {\em Unpublished}, 2005.

\bibitem[MO07]{marian_level-rank_2007}
Alina Marian and Dragos Oprea.
\newblock The level-rank duality for non-abelian theta functions.
\newblock {\em Inventiones mathematicae}, 168(2):225--247, January 2007.

\bibitem[MOP15]{marian_segre_2015}
Alina Marian, Dragos Oprea, and Rahul Pandharipande.
\newblock Segre classes and {Hilbert} schemes of points.
\newblock July 2015.
\newblock arXiv: 1507.00688.

\bibitem[MOP17]{marian_combinatorics_2017}
Alina Marian, Dragos Oprea, and Rahul Pandharipande.
\newblock The combinatorics of {Lehn}'s conjecture.
\newblock August 2017.
\newblock arXiv: 1708.08129.

\bibitem[MR10]{marangell_general_2010}
R.~Marangell and R.~Rimányi.
\newblock The general quadruple point formula.
\newblock {\em American Journal of Mathematics}, 132(4):867--896, August 2010.

\bibitem[Nak97]{nakajima_heisenberg_1997}
Hiraku Nakajima.
\newblock Heisenberg {Algebra} and {Hilbert} {Schemes} of {Points} on
  {Projective} {Surfaces}.
\newblock {\em Annals of Mathematics}, 145(2):379--388, 1997.

\bibitem[S{\etalchar{+}}17]{sage}
W.\thinspace{}A. Stein et~al.
\newblock {\em {S}age {M}athematics {S}oftware ({V}ersion 7.6)}.
\newblock The Sage Development Team, 2017.
\newblock {\tt http://www.sagemath.org}.

\bibitem[Sca07]{scala_dualite_2007}
Luca Scala.
\newblock Dualité étrange et cohomologie du schéma de {Hilbert}.
\newblock {\em Gazette des Mathématiciens}, (112):53--65, April 2007.

\end{thebibliography}

\end{document}